\documentclass{amsart}

\usepackage{amssymb,amsmath,comment,mathrsfs,mathtools,tikz-cd,todonotes,amscd}
\usepackage[colorlinks=true,linkcolor=blue]{hyperref}

\theoremstyle{remark}

\newtheorem{para}{\bf}[subsection]

\newtheorem{example}[para]{\bf Example}

\newtheorem{rem}[para]{\bf Remark}

\theoremstyle{definition}

\theoremstyle{plain}

\newtheorem{theorem}[para]{Theorem}

\newtheorem{lemma}[para]{Lemma}

\newtheorem{prop}[para]{Proposition}

\newenvironment{numequation}{\addtocounter{para}{1}
\begin{equation}}{\end{equation}}

\newenvironment{nummultline}{\addtocounter{para}{1}
\begin{multline}}{\end{multline}}

\mathchardef\mhyphen="2D

\newcommand{\bbA}{{\mathbb A}}
\newcommand{\bbB}{{\mathbb B}}

\newcommand{\bbF}{{\mathbb F}}
\newcommand{\bbG}{{\mathbb G}}
\newcommand{\bbH}{{\mathbb H}}

\newcommand{\bbN}{{\mathbb N}}

\newcommand{\bbP}{{\mathbb P}}
\newcommand{\bbQ}{{\mathbb Q}}
\newcommand{\bbR}{{\mathbb R}}

\newcommand{\bbU}{{\mathbb U}}

\newcommand{\bB}{{\bf B}}

\newcommand{\bF}{{\bf F}}
\newcommand{\bG}{{\bf G}}

\newcommand{\bZ}{{\bf Z}}

\newcommand{\frX}{{\mathfrak X}}

\newcommand{\cC}{{\mathcal C}}
\newcommand{\cD}{{\mathcal D}}

\newcommand{\cO}{{\mathcal O}}

\newcommand{\cR}{{\mathcal R}}

\newcommand{\cU}{{\mathcal U}}

\newcommand{\rig}{{\rm rig}}

\newcommand{\an}{{\rm an}}

\newcommand{\bksl}{\backslash}

\newcommand{\D}{{\Delta}}

\newcommand{\diag}{{\rm diag}}

\newcommand{\Fq}{{\bbF_q}}

\newcommand{\GL}{{\rm GL}}

\newcommand{\hra}{\hookrightarrow}
\newcommand{\id}{{\rm id}}
\newcommand{\im}{{\rm im}}
\newcommand{\Ind}{{\rm Ind}}

\newcommand{\la}{{\rm la}}

\newcommand{\lan}{\langle}

\newcommand{\lra}{\longrightarrow}
\newcommand{\midc}{{\; | \;}}
\renewcommand{\mod}{\; {\rm mod} \;}

\newcommand{\PD}{{\rm P{\Delta}}}
\newcommand{\oOmega}{{\overline{\Omega}}}
\newcommand{\oPOmega}{{\overline{\POmega}}}
\newcommand{\ori}{{\rm or}}

\newcommand{\oD}{\overline{\D}}
\newcommand{\opartial}{{\overline{\partial}}}
\newcommand{\oPD}{\overline{\PD}}

\newcommand{\POmega}{{ \rm P\Omega}}
\newcommand{\pr}{{\rm pr}}
\newcommand{\Proj}{{\rm Proj}}

\newcommand{\Qp}{{\bbQ_p}}
\newcommand{\ra}{\rightarrow}
\newcommand{\ran}{\rangle}
\newcommand\nonmini{\mathop{non\mhyphen min}}
\newcommand{\nonmin}{\rm \nonmini}
\newcommand{\res}{{\rm res}}

\newcommand{\setm}{{\; \setminus \;}}
\newcommand{\sgn}{{\rm sgn}}

\newcommand{\Stab}{{\rm Stab}}

\newcommand{\sub}{\subset}

\newcommand{\Sym}{{\rm Sym}}

\newcommand{\vpi}{{\varpi}}

\newcommand{\Z}{{\mathbb Z}}

    \usepackage{fancyhdr} 
     
\title{Resolutions of locally analytic principal series representations of $GL_2$}

\author{Aranya Lahiri}
\address{Department of Mathematics, Indiana University, Bloomington}
\curraddr{}
\email{lahiria@iu.edu}
\thanks{}

\begin{document}

\begin{abstract}
For a finite field extension $F/\Qp$ we associate a coefficient system attached on the Bruhat-Tits tree of $G:= \GL_2(F)$ to a locally analytic representation $V$ of $G$. This is done in analogy to the work of Schneider and Stuhler for smooth representations. This coefficient system furnishes a chain-complex which is shown, in the case of locally analytic principal series representations $V$, to be a resolution of $V$.

\end{abstract}

\maketitle
\tableofcontents

\section{Introduction}

\addtocounter{subsection}{1}

Let $F/\Qp$ be a finite field extension and $G = \bG(F)$ the group of $F$-valued points of a connected reductive group $\bG$ over $F$. In \cite{SS2} Peter Schneider and Ulrich Stuhler associated to a smooth representation $V$ of $G$ (over the complex numbers) coefficient systems on the Bruhat-Tits building $BT = BT(\bG)$ of $\bG$ (and also defined sheaves associated to $V$ on $BT$). These coefficient systems were shown to furnish resolutions of $V$. The purpose of this paper is to define analogous coefficient systems for locally analytic representations of $\GL_2(F)$ and to show that they too give rise to resolutions if the representation $V$ is a locally analytic principal series representation.  

\vskip8pt 

We quickly recall the coefficient system construction of Schneider and Stuhler for the general linear group, cf. \cite{SS1}.

\vskip8pt

For each vertex $v$ of $BT = BT(\GL_{n,F})$ let $G_v(0):= \Stab (v)\subset \GL_n(F)$, and let $G_v(k)$ be the $k^{\rm th}$ congruence subgroup of $G_v(0)$ for some $k\ge 1$. The coefficient system on $BT$  attached to a smooth representation $V$ of $\GL_n(F)$  is defined as follows

\begin{itemize}
\item To each vertex  $v \in BT$ we associate  $V_{v}:= V^{G_v(k)}$, the space of $G_v(k)$-fixed vectors of $V$.  And to each simplex $\sigma := \{v_0, v_1,\cdots, v_d\}$ we associate $V_\sigma = V^{G_\sigma(k)}$ where $G_\sigma(k) :=\lan G_{v_0}(k), \cdots, G_{v_d}(k) \ran$, is the group generated by $G_{v_i}(k)$'s. 
\vskip5pt
\item If $\sigma \subseteq \tau$ are two simplices, then $G_\sigma(k) \sub G_\tau(k)$, and there is hence an inclusion $r_\sigma^\tau: V_\tau \ra V_\sigma$. These transition maps satisfy the conditions $r^{\sigma}_{\sigma}= \id$, and , $r_\sigma^\tau \circ r_\tau^\xi= r_\sigma^\xi$ for simplices $\sigma \subseteq \tau \subseteq \xi$.
\end{itemize}

This coefficient system naturally gives rise to a  chain-complex. One  of the key results of \cite{SS1, SS2} is that this complex is exact.

\vskip8pt 

In section \ref{coeffcomplex} of this paper we associate an analogous coefficient system to a locally analytic representation $V$ of $G:= \GL_2(F)$, replacing $G_v(k)$-fixed vectors by the rigid-analytic vectors $V_{\bbG_v(k)-\an}$ for the corresponding rigid analytic group $\bbG_v(k)$ and construct a chain complex [\ref{chaincomp}] associated to it. In fact, the construction of the coefficient system and the associated chain-complex can be generalized to locally analytic representations of any (connected) $p$-adic reductive group. It is natural to ask, when is this chain complex a resolution of $V$?

\vskip8pt
 In section \ref{resolution} we show that  under some assumptions on $k$,

\begin{theorem}[see \ref{main}]
For the locally analytic principal series $V:= Ind^G_B(\chi)$, the chain-complex \ref{chaincomp} is a resolution of $V$.

\end{theorem}

\vskip8pt

{\it Notation.} We denote by $F/\Qp$ a finite field extension of $\Qp$, by $\cO$ its ring of integers, by $\vpi$ a uniformizer, and by $\Fq$ its residue field of cardinality $q$. We will denote by $E/F$ a finite field extension (the `coefficient field'), and all locally analytic representations will be over $E$. Throughout this paper $G$ will denote the group $\GL_2(F)$. For a locally $F$-analytic manifold $X$, we denote by $C^\la(X, E)$  the the space of $E$-valued  locally  $F$-analytic functions on $X$ as defined in \cite{S-T} (in the reference, the authors use $C^\an(X, E)$ for this space).

\vskip8pt

\section{A coefficient system and complex} \label{coeffcomplex}

\subsection{The Bruhat-Tits tree and associated subgroups} 

\begin{para}{\it The building.} Recall that the semisimple Bruhat-Tits building $BT$ of $\GL_2$ over $F$ is a one-dimensional simplicial complex whose set of vertices $BT_0$ we can identify with the set of homothety classes of $\cO$-lattices $\Lambda \sub F^2$. We write $[\Lambda]$ for the homothety class of $\Lambda$. Two vertices $v$ and $v'$ are adjacent (or form an edge) if and only if there are representing lattices $\Lambda$ of $v$ and $\Lambda'$ of $v'$, such that $\vpi\Lambda \subsetneq \Lambda' \subsetneq \Lambda$. The set of edges of $BT$ will be denoted by $BT_1$, and we write $e = \{v,v'\}$ if the vertices $v$ and $v'$ form an edge. We define the distance function $d: BT_0 \times BT_0 \ra \Z_{\ge 0}$ as follows: $d(v,v') = 0$, if $v=v'$, and $d(v,v') = n$, if there is a sequence of vertices $v = v_0, v_1, v_2, \ldots, v_n = v'$ such that $\{v_i,v_{i+1}\}$ is an edge for all $i \in \{0,\ldots,n-1\}$, and $v_{i+1} \neq v_{i-1}$ for $i \in \{1, \ldots, n-1\}$. We further recall that $BT$ is a homogeneous tree of degree $q+1$, in particular the distance function is well-defined.

\vskip8pt

An {\it oriented edge} is an ordered pair $(v,v')$ of adjacent vertices. An {\it orientation} of $BT$ is a set $BT_1^\ori \sub BT_0 \times BT_0$ consisting of oriented edges, and such that the map $BT_1^\ori \ra BT_1$, $(v,v') \mapsto \{v,v'\}$, is bijective.  For a vertex $v$ and an edge $e = \{v_1,v_2\}$, or an oriented edge $e = (v_1,v_2)$, we set $d(e,v) = \max\{d(v_1,v),d(v_2,v)\}$.

\vskip8pt

The group $G = \GL_2(F)$ acts on $BT$ by $g.[\Lambda] = [g.\Lambda]$, and by $g.\{v,v'\} = \{g.v, g.v'\}$. The action of $G$ on $BT_0$ and $BT_1$ are well known to be transitive. In \ref{containment} we will need the following transitivity of $G$-action.
\end{para}

\begin{lemma}\label{n-transitive}
Let $v, v'$ and $w, w'$ be two pairs of vertices with $d(v,v') = d(w, w') = n\geq 0$, and let $v = v_0, v_1, v_2, \ldots, v_n = v'$ and $w = w_0, w_1, w_2, \ldots, w_n = w'$ be the unique paths connecting $v$ with $v'$ and $w$ with $w'$, respectively. Then there exists $g \in G$ such that for all $i \in \{0, \ldots, n\}$ one has $g.v_i= w_i$, i.e., the action of $G$ on paths of length $n \ge 0$ is transitive. 
\end{lemma}

\begin{proof} The proof proceeds by induction on $n$. The assertion is true in the case when $n \le 1$ because of the transitivity of the $G$-action on $BT_0$ and $BT_1$, as mentioned above. We thus assume that $n \ge 2$ in the following. Furthermore, it is enough to prove the assertion under the assumption that $v_0 = [\cO\oplus \cO]$ and $v_i= [(\vpi^i) \oplus \cO]$ for $0 \le i \le n$. By induction, there this $h \in G$ such that $h.v_i = w_i$ for $0 \le i \le n-1$. Then $h^{-1}.w_n$ is adjacent to $v_{n-1} = h^{-1}.w_{n-1}$ and is different from $v_{n-2}$. The vertices which have this property are of the form $v_{\alpha,n}:=[\lan (\vpi^n,0), ([\alpha],1)\ran]$, the homothety class of  $\cO$-lattice generated by $(\vpi^n,0)$ and  $([\alpha],1)$ where $\alpha \in (\vpi)^{n-1}/(\vpi)^n$. Then $h' := \begin{bmatrix}1 & \alpha \\ 0 & 1\end{bmatrix}$ takes $v_{\alpha,n}$ to $v_n$ while fixing $v_i$ for $0\le i \le n-1$. The element $g = h(h')^{-1}$ has then the desired property. 
\end{proof}

\begin{para}{\it The groups $G_\sigma(k)$.}  Given a vertex $v = [\Lambda]$ we set $G_v(0) = \Stab_G(\Lambda) = \{g \in G \midc g.\Lambda = \Lambda\}$, and for a positive integer $k$ we put

$$G_v(k) = \{g \in G_v(0) \midc \forall\, x \in \Lambda: \; g.x \equiv x  \mod \vpi^k\Lambda\}$$

\vskip8pt

Given an edge $e = \{v,v'\}$ (or an oriented edge $e = (v,v')$) and $k \in \Z_{>0}$ we let $G_e(k) = \lan G_v(k), G_{v'}(k) \ran$ be the subgroup of $G$ generated by $G_v(k)$ and $G_{v'}(k)$. As $G$ acts transitively on $BT_0$ and $BT_1$, $G$ acts transitively on $\{G_v(k) \midc v \in BT_0\}$ as well as on $\{G_e(k) \midc e \in BT_1\}$ by conjugation. 
\end{para}

For later purposes it will be useful to describe some of the groups $G_v(k)$ and $G_e(k)$ explicitly. 

\begin{lemma}\label{simplexgroups} Fix $k \in \Z_{>0}$. If $v_0 = [\cO \oplus \cO]$, then

\begin{numequation}\label{Gv}
G_{v_0}(k) = \left\{\begin{bmatrix} 1+ \vpi^k a & \vpi^k  b\\ \vpi^k c & 1+\vpi^k d\end{bmatrix} \Big| \; a,b,c,d \in  \cO \right\} \;,
\end{numequation}

and if $e = \{v_0,v_1\}$ with $v_1 = [(\vpi) \oplus \cO]$, then 

\begin{numequation}\label{Ge}
G_e(k) =\left \{\begin{bmatrix} 1 + \vpi^k a & \vpi^{k} b \\ \vpi^{k-1} c & 1+ \vpi^k d \end{bmatrix}\Big| \; a, b, c, d \in \cO \right\} \;.
\end{numequation}
\end{lemma}

\begin{proof} The first assertion is well-known and straightforward to verify. For the second assertion we note that $v_1 = \diag(\vpi,1).v_0$ and therefore 

$$G_{v_1}(k) = \diag(\vpi,1)G_{v_0}(k)\diag(\vpi,1)^{-1} = \left\{\begin{bmatrix} 1+ \vpi^k a & \vpi^{k+1}  b\\ \vpi^{k-1} c & 1+\vpi^k d\end{bmatrix} \Big| \; a,b,c,d \in  \cO \right\} \;,$$

By definition, $G_e(k)= \lan G_{v_0}(k), G_{v_1}(k)\ran$. Let $H$ be the group on the right hand side of \ref{Ge}. An easy computation shows that 

\vskip8pt

(i) $G_{v_0}(k). G_{v_1}(k) = G_{v_1}(k). G_{v_0}(k)$, and hence $G_e(k) = G_{v_0}(k). G_{v_1}(k)$, and that 

\vskip5pt

(ii) $G_{v_0}(k). G_{v_1}(k) \sub H$.

\vskip8pt

Now we observe that for any  $a, b, c, d \in \cO_F$ we have 

\[\begin{bmatrix} 1 + \vpi^k a & \vpi^{k} b \\ \vpi^{k-1} c & 1+ \vpi^k d \end{bmatrix}\; =  \begin{bmatrix} 1 + \vpi^k a & 0 \\ \vpi^{k-1}c & 1 \end{bmatrix}. \begin{bmatrix} 1 & \vpi^k \frac{b}{1+\vpi^k a} \\ 0 & 1+ \vpi^k \Big(d-\vpi^{k-1}\frac{cb}{1+\vpi^ka}\Big) \end{bmatrix}\]

This finishes the proof.
\end{proof}

\begin{para}{\it A Coefficient System on $BT$.}\label{locancoeff} 
We choose $k \in \Z_{>0}$ such that for any simplex (edge or vertex) $\sigma$ of $BT$ the group $G_\sigma(k)$ as defined above is a  uniform pro-$p$ group \cite{DDMS}. Under this assumption we can associate to $G_\sigma(k)$ a rigid analytic subgroup $\bbG_\sigma(k)$ of the analytification $\GL_2^\rig$ of the algebraic group $\GL_2$ over $F$. See \cite[3.5]{La} for a description of this process. 

\vskip8pt

Given a locally analytic representation $V$  of $G$, we define a coefficient system $\cC^{(k)}_V$ on $BT$ as follows.

\vskip8pt

\begin{itemize}
\item To each simplex $\sigma \sub BT$ we associate  $V_{\sigma}:= V_{\bbG_\sigma(k)-\an}$, the space of $\bbG_\sigma(k)$-analytic vectors\footnote{See \cite[3.3.13]{Em17} for the definition of analytic vectors of a representation with respect to a good analytic subgroup.} of $V$ . 

\vskip5pt

\item If $\sigma \subseteq \tau$ are two simplices, then $G_\sigma(k) \sub G_\tau(k)$, and there is hence an inclusion $r_\sigma^\tau: V_\tau \ra V_\sigma$. These transition maps satisfy the conditions $r^{\sigma}_{\sigma}= \id$, and therefore, $r_\sigma^\tau \circ r_\tau^\xi= r_\sigma^\xi$ for simplices $\sigma \subseteq \tau \subseteq \xi$.

\end{itemize}

\end{para}

\subsection{The chain complex associated to a coefficient system on \texorpdfstring{$BT$}{}}

We associate to $\cC^{(k)}_V$ the following complex

\begin{numequation}\label{chaincomp}
0\ra \bigoplus_{e \in BT_1^\ori} V_e \stackrel{\partial_1}{\longrightarrow} \bigoplus_{v\in BT_0} V_v \stackrel{\partial_0}{\longrightarrow} V \ra 0
\end{numequation}

Where there maps are defined as follows: given an oriented edge $e= (v_1, v_2) \in BT_1^\ori$ the map $\partial_1|_{V_e}$ is defined by 

\[\begin{array}{ccc}
    \partial_1|_{V_e} : V_e  & \lra & \bigoplus_{v\in BT_0} V_v\\ f & \longmapsto & (f_v)_v \;,
\end{array}\]

\vskip8pt

where $f_{v_1} = f$, $f_{v_2} = -f$, and $f_v = 0$ for all $v \not \in \{v_1,v_2\}$. The map $\partial_0$ is defined by 

\begin{align*}
    \partial_0 : \bigoplus_{v\in BT_0} V_v &\longrightarrow V \\ (f_v)_v &\ra \sum_v f_v \;,
\end{align*}

where the sum is taken inside $V$ using the vector space embeddings $V_v \hookrightarrow V$.

\section{Locally analytic principal series representations}\label{resolution}

\subsection{Locally analytic induction} 

\begin{para}\label{inducedrep} Let $T  \subset G$ be the the maximal torus comprising of diagonal matrices and B with, $T\subset B \subset G$ be the standard Borel subgroup of upper triangular matrices. We also fix once and for all a finite extension $E/F$. For a locally $F$-analytic character $\chi: T \ra E^\times$ we consider the locally analytic principal series representation

\[\begin{array}{rcl} V & := & \Ind^G_B(\chi) \\
&&\\
& = & \Big\{f: G \ra E \mbox{ locally $F$-analytic} \Big |\, \forall \, g \in G \; \forall\; b \in B: \; f(bg) = \chi(b)f(g)\Big\} \;.
\end{array}\]

\vskip8pt

The action of $G$ on $V$ is given by $g.f(x)= f(x. g^{-1})$. We will consider $\Ind^G_B(\chi)$ as a topological vector space as follows. Set $G_0 = \GL_2(\cO)$ and $B_0 = B \cap G_0$. Then $G = B \cdot G_0$ and the canonical map of quotients $B_0 \bksl G_0 \ra B \bksl G$ is an isomorphism of locally $F$-analytic manifolds. Therefore, restricting locally $F$-analytic functions from $G$ to $G_0$ gives an isomorphism of vector spaces $\Ind^G_B(\chi) \ra \Ind^{G_0}_{B_0}(\chi)$. Since $G_0$ is compact, the space $C^\la(G_0,E)$ of $E$-valued locally $F$-analytic functions on $G_0$ naturally carries the structure of an $E$-vector space of compact type \cite[Lemma 2.1]{S-T}. We equip $\Ind^G_B(\chi)$ with the structure of a locally convex $E$-vector space so that the map $\Ind^G_B(\chi) \ra \Ind^{G_0}_{B_0}(\chi)$ becomes an isomorphism of topological vectors spaces. 
\end{para}

\begin{para} For a simplex $\sigma \in BT$ (or an oriented edge $e \in BT_1^{\ori}$), let 

\[\Omega_{\sigma, k}= B\backslash G/G_\sigma(k)\]

\vskip8pt

be the set of $B$-$G_\sigma(k)$ double cosets in $G$, which is finite because $B \bksl G$ is compact and $G_\sigma(k)$ open in $G$. As $G_\sigma(k)$ is an open subgroup, so is any double coset $\D \in \Omega_{\sigma,k}$. Given  $\D \in \Omega_{\sigma,k}$, we set 

\[I(\D, \chi)= \Big\{\D \stackrel{f}{\lra} E \;\Big| \; f \mbox{ is  loc. $F$-analytic}, \, \forall\; b \in B, x \in \D, f(bx)= \chi(b)f(x)\Big\}\]

\vskip8pt

Note that by extending functions in $I(\D,\chi)$ by zero outside $\D$, we obtain an embedding $I(\D,\chi) \hra V$, and the image of this map, which we will henceforth identify with $I(\D,\chi)$, is stable under the action of $G_\sigma(k)$ on $V$. Because $G = \coprod_{\D \in \Omega_{\sigma,k}} \D$, we have 

\[V_{\bbG_{\sigma}(k)-\an} = \oplus_{\D \in \Omega_{\sigma, k}}I(\D,\chi)_{\bbG_{\sigma}(k)-\an} \;.\]

\end{para}

The main result of this article is

\begin{theorem}\label{main}
There exists an integer $k_0 \ge 1$ such that for all $k\ge k_0$ the chain complex \ref{chaincomp} is a resolution when applied to $V= \Ind^G_B(\chi)$.
\end{theorem}

The rest of the article is dedicated to the proof of \ref{main}. A crucial role in our proof will be played by the following sets of vertices and edges:

\[\begin{array}{lclclc}
    BT_{0,n} & = &\Big\{v\in BT_0 \; \Big| \; d(v, v_0)\leq n\Big\} \;,\\
    &&\\
    BT^\ori_{1,n} & = & \Big\{e=(v,v') \in BT^\ori_1 \; \Big| \; d(e,v_0) \leq n \Big\}
\end{array}\]
 
 \vskip8pt
where $v_0 = [\cO \oplus \cO]$ is as above.

\subsection{Injectivity of \texorpdfstring{$\partial_1$}{} }

\begin{prop} \label{injective}
The map $\partial_1: \bigoplus_{e\in BT_1^\ori}V_e \longrightarrow \bigoplus_{v\in BT_0}V_v $ is injective.
\end{prop}

\begin{proof}
Let  $f = (f_e)_e \neq 0$ be supported on $BT_{1,n}^\ori$ but not on $BT_{1,n-1}^\ori$ for some $n\geq 1$ ($BT_{1,0}^\ori$ is the empty set), i.e., there is $e = (v_1,v_2) \in BT^\ori_1$ such that $d(e,v_0) = n$ and $f_e \neq 0$, and $f_{e'} = 0$ for all edges $e'$ with $d(e',v_0) > n$. If then $i \in \{1,2\}$ is such that $d(v_i,v_0) = n$, we have $\partial_1(f)_{v_i} = \pm f_e \neq 0$.
\end{proof}

\subsection{The action of \texorpdfstring{$G$}{} on \texorpdfstring{$\bbP^1(F)$}{}}\label{Gaction}

In the rest of the paper we will let $\infty$ be a symbol different from all elements in $F$ and put $\bbP^1(F) = F \cup \{\infty\}$. We equip $\bbP^1(F)$ with an action from the right by $G$ via M\"obius transformations. Explicitly, the action of $g= \begin{bmatrix} a & b \\c & d\end{bmatrix} \in G$ on $z \in \bbP^1(F)$ is given by 

\begin{numequation}\label{action} z.g = \left\{\begin{array}{ccl} \frac{az + c}{bz +d} &, & z \neq \infty \\
\frac{a}{b} & ,&  z = \infty \end{array}\right.
\end{numequation}

where $\frac{az + c}{bz +d}$ (resp. $\frac{a}{b}$) is $\infty$ if the denominator vanishes. The stabilizer of the point $0 \in \bbP^1(F)$ is $B$, and the map

\[\iota: B \bksl G \lra \bbP^1(F) \;,\;\; Bg \mapsto 0.g \;,\]

\vskip8pt

is bijective and $G$-equivariant, if we consider the right translation action of $G$ on $B \bksl G$. The quotient $B \bksl G$ inherits the structure of a locally $F$-analytic manifold from $G$, and we equip $\bbP^1(F)$ with the structure of a locally $F$-analytic manifold via $\iota$ (so that $\iota$ becomes an isomorphism of locally $F$-analytic manifolds). Each $g \in G$ acts as a locally $F$-analytic automorphism on $\bbP^1(F)$. 

\vskip8pt

We denote by $\POmega_{\sigma,k}$ the set of $G_\sigma(k)$-orbits on $\bbP^1(F)$. By the discussion in the preceding paragraph, there is a canonical bijection  between $B$-$G_\sigma(k)$ double cosets of $G$ and orbits of (right) action of $G_\sigma(k)$ on $\bbP^1(F)$, given by 
\[\Omega_{\sigma, k} \stackrel{1:1}{\lra} \POmega_{\sigma,k} \;, \; \D \mapsto \PD := B\bksl \D \;.\]
 
 \vskip8pt
For $z_0 \in F \sub \bbP^1(F)$ and $r \in \bbR_{> 0}$ we then set $\bB_z(z_0, r) = \{x \in F \midc |x-z_0| \le r\}$, which is a closed disc of radius $r$ around $z_0$. Similarly, for $r \in \bbR_{> 0}$ and $w_0 \in \bbP^1(F) \setminus \{0\}$ we set $\bB_w(w_0,r) = \{x \in F^\times \cup \{\infty\} \midc \left|\frac{1}{x}- \frac{1}{w_0}\right| \le r \}$, where we interpret $\frac{1}{\infty}$ as zero. In particular, $\bbP^1(F) = \bB_z(0,1) \sqcup \bB_w(\infty,|\vpi|)$.

\begin{prop} \label{particular orbits}  For $v_0 =[\cO \oplus \cO], v_1= [(\vpi)\oplus \cO]$ and $e_0= \{v,v_1\}$, the orbits of $G_{v_0}(k), G_{v_1}(k)$ and $G_{e_0}(k)$ can be described as follows

\vskip8pt

(i) The orbits of $G_{v_0}(k)$ on $\bbP^1(F)$ are discs of radius $|\vpi|^k$ of the form $\bB_z(z_0, |\vpi|^k)$ on $\bB_z(0, 1)$ with $z_0\in \bB_z(0, 1)$ and are of the form $\bB_w(w_0, |\vpi|^k)$ on $\bB_w(\infty, |\vpi|)$ with $w_0\in \bB_w(\infty, |\vpi|)$. 

\vskip8pt 

(ii)The orbits of $G_{v_1}(k)$ are discs of radius $|\vpi|^{k-1}$ of the form $\bB_z(z_0, |\vpi|^{k-1}$ on $\bB_z(0,1)$ with $z_0\in \bB_z(0,1)$ and discs of radius $|\vpi|^{k+1}$ of the form $\bB_w(w_0, |\vpi|^{k+1})$ on $\bB_w(\infty, |\vpi|)$ with $w_0\in \bB_w(\infty, |\vpi|) $. 
\vskip8pt

(iii) The orbits $G_{e_0}(k)$-orbits on $\bB_z(0,1)$ are of the form $\bB_z(z_0, |\vpi|^{k-1})$ (i.e.,  same as orbits of of $G_{v_1}(k)$) with $z_0\in \bB_z(0,1)$ and on $\bB_w(\infty, |\vpi|)$ they are of the form $\bB_w(w_0, |\vpi|^{k})$ (i.e., same as orbits of of $G_{v_0}(k)$) with $w_0 \in \bB_w(\infty, |\vpi|)$ . 
\end{prop}

\begin{proof}

(i) and (ii) To compute the orbits of $G_{v_0}(k)$ on $\bB_z(0,1)$ we use the description of $G_{v_0}(k)$ from \ref{Gv} and the description of the action in \ref{action} to get for $z_o\in \bB_z(0,1)$

\begin{align*}
    \frac{az_0 + c}{bz_0 +d}= (z_0 + O(\vpi^k))(1 + O(\vpi^k))= z_0+ O(\vpi^k)
\end{align*}

for $g= \begin{bmatrix} a & b \\c & d\end{bmatrix} \in G_{v_0}(k)$. This shows that the orbit of $z_0$ is contained in $\bB_z(z_0, |\vpi|^k)$. Via the translation $z_0 \ra z_0.\begin{bmatrix} 1 & 0\\ c & 1\end{bmatrix} = z_0 + c$, we conclude that the orbit is indeed $\bB_z(z_0, |\vpi|^k)$. A similar computation shows that the orbit of $w_0 \in \bB_w(\infty, |\vpi|)$ is $\bB_w(w_0, |\vpi|^k)$. 

\vskip8pt

For $g = \begin{bmatrix} a & b \\c & d\end{bmatrix} \in G_{v_1}(k)$ we have,
\begin{align*}
    \frac{az_0 + c}{bz_0 +d}= (z_0 + O(\vpi^{k-1}))(1 + O(\vpi^k))= z_0+ O(\vpi^{k-1})
\end{align*}
And again by using matrices of the form $\begin{bmatrix} 1 & 0\\ c & 1\end{bmatrix}$ we obtain that the orbits on $\bB_z(0,1)$ are $\bB_z( z_0, |\vpi|^{k-1})$.
On the other hand, on $\bB_w(\infty, |\vpi|)$ we have, \begin{align*}
    \frac{1}{w_0.g} = \frac{bw_0+ d}{aw_0+c} = \frac{b+ d\frac{1}{w_0}}{a+c\frac{1}{w_0}} = (\frac{1}{w_0}+ O(\vpi^{k+1}))(1 + O(\vpi^k))= \frac{1}{w_0}+ O(\vpi^{k+1})
\end{align*}

thus the orbits on $\bB_w(\infty, |\vpi|)$ of $G_{v_1}(k)$ are disjoint discs $\bB_w(w_0,|\vpi|^{k+1})$ and $\bB_w(\infty, |\vpi|^{k+1})$. This shows assertion (ii). 
\vskip8pt

 (iii) Using \ref{Ge}, we can show via an analogous computation that the orbits of $G_{e_0}(k)$ on $\bB_z(0,1)$ are of the form $\bB_z(z_0, |\vpi|^{k-1})$ (i.e., orbit of of $G_{v_1}(k)$) and on $\bB_w(\infty, |\vpi|)$ they are of the form $\bB_w(w_0, |\vpi|^{k}), w_0 \in \bB_w(\infty, |\vpi|) $ (i.e., orbits of  $G_{v_0}(k)$). 
\end{proof}

We can make a more general statement based  on the previous proposition.

\begin{prop}\label{orbit}
Let $v, v'\in BT_0$ and $e = \{v,v'\} \in BT_1$. We have,
\vskip8pt

(i) $\Omega_{v,k}\cap \Omega_{v',k} = \emptyset$.

\vskip8pt

(ii) For every $\D \in \Omega_{v,k}$ there is $\D' \in \Omega_{v',k}$ such that $\D \sub \D'$ or $\D' \sub \D$. Furthermore, $\Omega_{v,k}^{v'} := \{\D \in \Omega_{v,k} \midc \exists \D' \in \Omega_{v',k}: \D' \sub \D\}$ has cardinality $q^{k-1}$ and each $\D \in \Omega_{v,k}^{v'}$ contains $q$ orbits of $\Omega_{v',k}$.

\vskip8pt

(iii)  If $\D \in \Omega_{e,k}$ then $\D \in \Omega_{v,k}$ or $\D \in \Omega_{v',k}$.

\vskip8pt

(iv) If $\D \in \Omega_{v,k} \cap \Omega_{e,k}$ then there exists $\D'\in \Omega_{v',k}$ such that $\D' \subset \D$ and similarly, if $\D' \in \Omega_{v',k} \cap \Omega_{e,k}$ then there exists $\D\in \Omega_{v,k}$ such that $\D \subset \D'$.

\vskip8pt

(v) If $\D'\subset \D$ for double-cosets in $\Omega_{v',k}$ and $\Omega_{v,k}$, respectively, then $\D \in \Omega_{v,k}\cap \Omega_{e,k}.$
\end{prop}
\begin{proof} Via the canonical  bijection  $\D \leftrightarrow \PD$ it is enough to show the corresponding assertions for orbits on $\bbP^1(F)$.

\vskip8pt

 Since the action of $G$ on  $BT_0$ and $BT_1$ is transitive  it is enough to prove the statements for $v_0, v_1$ and $e_0 = \{v_0,v_1 \}$, where $v_0$ and $v_1$ are as in \ref{particular orbits}. 

\vskip8pt
(i), (ii) and (iii) Follows directly from \ref{particular orbits}.

\vskip8pt 

(iv) If $\PD=\bB_w(w_0, |\vpi|^k)\subset \bB_w(\infty, |\vpi|) \in \POmega_{v_0,k}\cap \POmega_{e_0,k}$ then we can take $\PD'= \bB_w(w_0, |\vpi|^{k+1}) \in \Omega_{v_1,k}$. And if $\PD=\bB_z(z_0, |\vpi|^{k-1}) \subset \bB_z(0,1) \in \POmega_{v_1,k}\cap \POmega_{e_0,k}$ then we can take $\PD' = \bB_z(z_0, |\vpi|^k)$.

\vskip8pt

(v) The only orbits for which the relation $\PD' \subset \PD$ with $\PD' \in \POmega_{v_1,k}, \PD \in \POmega_{v_0,k}$ holds are of the form $\bB_w(w_0, |\vpi|^{k+1}) \subset \bB_w(w_0, |\vpi|^k)$. And any such $\bB_w(w_0, |\vpi|^k) \in \POmega_{v_0,k}$ also belongs to $\POmega_{e_0, k}$.
\end{proof}

\begin{rem}\label{cosetbijection}
Let $\sigma', \sigma$ be two simplices in $BT$ with $\sigma'= g.\sigma$ for $g\in G$. The map

\[\begin{array}{cccccc}
   \Omega_{\sigma,k}& \lra & \Omega_{\sigma',k},\;\;\;\;\;\; \POmega_{\sigma,k}& \lra & \POmega_{\sigma',k}  \\
    \D  & \lra & \D. g^{-1}, \;\;\;\;   \PD  & \lra & \PD. g^{-1}
\end{array}\] defines a bijection between   $\Omega_{\sigma,k}$ and  $\Omega_{\sigma',k}$ and between $\POmega_{\sigma,k}$ and  $\POmega_{\sigma',k}$ respectively.
\end{rem}

\begin{lemma}\label{n=1}
Let $v_{\alpha,1}:= [\lan (1,\alpha), (0,\vpi) \ran]$, $\alpha \in \cO/(\vpi)$, be the vertex corresponding to the homothety class of the $\cO$-lattice $\lan (1,[\alpha]), (0,\vpi) \ran \sub F^{\oplus 2}$. Then the orbits of $G_{v_{\alpha,1}}(k)$ on $\bB_z([\alpha], |\vpi|) \subset \bB_z(0,1)$ are of the form $\bB_z(\beta, |\vpi|^{k+1})$, for $\beta \in \bB_z([\alpha], |\vpi|)$. Moreover each of the orbits of $G_{v_{\alpha, 1},k}$ of the form $\bB_z(\beta, |\vpi|^{k+1})$ on $\bB_z([\alpha], |\vpi|)$ is contained in an orbit of the form $\bB_z(z_0, |\vpi|^k)$ of $G_{v_0,k}$. 

\vskip8pt 
Each of the other orbits of $G_{v_{\alpha,1}}(k)$ contains $q$ orbits of $G_{v_0}(k)$. 
\end{lemma}

\begin{proof}
Note that $v_{\alpha,1} = g_\alpha. v_0$, where $g_\alpha= \begin{bmatrix} 1 & 0 \\ [\alpha] & \vpi\end{bmatrix}$ and $v_0= [\cO \oplus \cO]$ as in \ref{orbit}. Thus by \ref{cosetbijection} we have $\POmega_{v_{\alpha, 1},k}= \POmega_{v_0,k}. g_\alpha^{-1}$. Now $$\beta.g_\alpha^{-1}= \beta. \begin{bmatrix} 1 & 0 \\ -\frac{[\alpha]}{\vpi} & \vpi^{-1}\end{bmatrix}= \frac{\beta-\frac{[\alpha]}{\vpi}}{\vpi^{-1}}= \vpi.\beta- [\alpha].$$ This transformation takes $\bB_z(0, |\vpi|^k). g_\alpha^{-1}= \bB_z([\alpha], |\vpi|^{k+1})$ and very similarly we can compute the other orbits. The containment relation is clear from the radius of the  respective orbits.
\end{proof}

\begin{para}\label{Omega-ks}  We also introduce the following notation for the rest of the paper

\[\begin{array}{lclclclclclc}

\Omega_{0, k} &=& \bigsqcup_{v \in BT_0} \Omega_{v,k}\;,& \; \POmega_{0, k} & = &\bigsqcup_{v \in BT_0} \POmega_{v,k}\;\\
    &&\\
\Omega_{0, k,n} & =&  \bigsqcup_{v \in BT_{0,n}} \Omega_{v,k}\;,& \;\POmega_{0, k,n} & = & \bigsqcup_{v \in BT_{0,n}} \POmega_{v,k}\;\\
    &&\\    
     \Omega_{1,k}  & =&   \bigsqcup_{e\in BT_1^{\ori}}\Omega_{e,k}\;,&\; \POmega_{1,k} & = &   \bigsqcup_{e\in BT_1^{\ori}}\POmega_{e,k}\\
    &&\\
\Omega_{1, k,n} & =&  \bigsqcup_{e \in BT^{\ori}_{1,n}} \Omega_{e,k}\;, & \;\POmega_{1, k,n} & = & \bigsqcup_{e \in BT^{\ori}_{1,n}} \POmega_{e,k}\\&&\\
\oOmega_{0, k} &=& \bigcup_{v \in BT_0} \Omega_{v,k}\;,& \; \oPOmega_{0, k} & = &\bigcup_{v \in BT_0} \POmega_{v,k}\;\\
    &&\\
\oOmega_{0, k,n} & =&  \bigcup_{v \in BT_{0,n}} \Omega_{v,k}\;,& \;\oPOmega_{0, k,n} & = & \bigcup_{v \in BT_{0,n}} \POmega_{v,k}\;\\
    &&\\    
     \oOmega_{1,k}  & =&   \bigcup_{e\in BT_1^{\ori}}\Omega_{e,k}\;,&\; \oPOmega_{1,k} & = &   \bigcup_{e\in BT_1^{\ori}}\POmega_{e,k}\\
    &&\\
\oOmega_{1, k,n} & =&  \bigcup_{e \in BT^{\ori}_{1,n}} \Omega_{e,k}\;, & \;\oPOmega_{1, k,n} & = & \bigcup_{e \in BT^{\ori}_{1,n}} \POmega_{e,k}\;
\end{array}\]

\vskip8pt

\begin{rem}\label{subsetrelations}

Note that, (iii) of \ref{orbit} says, If $\D \in \Omega_{e,k}$ for $e=\{v,v'\}$ then $\D \in \Omega_{v,k}$ or $\D \in \Omega_{v',k}$. From this it is clear that we can think of $\Omega_{1,k}$ as a subset of  $\Omega_{0,k}$. And we can similarly conclude $\POmega_{1,k} \subset \POmega_{0,k}, \oOmega_{1,k} \subset \oOmega_{1,k}$ and $\oPOmega_{1,k} \subset \oPOmega_{0,k}$. 

\vskip8pt

We equip $\oOmega_{0,k}$ (and $\oPOmega_{0,k}$) with the partial order via inclusion.
\end{rem}
 
\end{para}

\begin{rem}
Note that $\Omega_{0,k,n}$ and $\oOmega_{0,k,n}$ (and by extension $\Omega_{0,k}$ and $\oOmega_{0,k}$)  are not the same. For example,  by \ref{n=1}, $\bB_w(\infty, |\vpi|)$ is in both $\Omega_{v_{1,1},2}$ and $\Omega_{v_{0,1},2}$, where $v_{0,1}$ and $v_{1,1}$ are as in \ref{n=1}. There is an obvious map $\pr: \Omega_{0,k} \ra \oOmega_{0,k} $ which restricts to a map from $\pr: \Omega_{0,k,n} \ra \oOmega_{0,k,n} $, taking every orbit in $\oOmega_{0,k}$  that are same as sets to the same set in $\oOmega_{0,k}$. We will use the same notation for the map $\pr: \POmega_{0,k} \ra \oPOmega_{0,k}$. 

\end{rem}

 For sake of clarity, from now on let $\oD = \pr(\D)$ and  $\oPD = \pr(\PD)$. When we want to emphasize that the double coset $\D \in \Omega_{v,k}$ (resp. the orbit $\PD \in \POmega_{v,k}$) is just considered as a subset of $G$ (resp. a subset of $\bbP^1(F)$), then we write $\oD$ (resp. $\oPD$) for it.

\begin{prop}\label{discsorbits}
For $G_v(k)$-orbits on $\bbP^1(F)$ we have the following:

\vskip8pt

(i) Given any $z \in \bbP^1(F)$ and $k \in \Z_{>0}$, the set \[\{\oPD \midc  v \in BT_0\,, \; \PD \in  \POmega_{v,k}\,,\; z \in \overline{\PD}\}\] 
is a fundamental system of open compact neighborhoods of $z$.

\vskip8pt 
(ii)\label{refine} Given any covering $\cU$ of $\bbP^1(F)$ there is a  finite subset of $\oPOmega_{0,k}$ that refines  $\cU$.

\end{prop}

\begin{proof} 
(i) Without loss of generality we may assume that $z=0$.
Consider the vertices $v_n := [\cO \oplus (\vpi^n)], n \ge 1$. We have $g_n. v_0 = v_n$, where $g_n = \begin{bmatrix} 1 & 0 \\ 0 & \vpi^n\end{bmatrix}$. Let $\PD_0 := \bB_z( 0, |\vpi|^k) \in \POmega_{v_0,k}$, then we can compute that $\PD_n:= \PD_0. g_n^{-1} = \bB_z(0, |\vpi|^{n+k}) \in \POmega_{v_n,k}$. And, $\{\oPD_n\}_{n\ge 1} \subset \{\oPD \in \POmega_{v,k} \midc v \in BT_0, z \in \oPD\}$ forms a fundamental neighborhood around $z=0$ of compact open subsets of $\bbP^1(F)$.

\vskip8pt

(ii) This is an immediate corollary of (i).
\end{proof}

\subsection{Containment relations between the orbits}

\begin{para}{\it Another description of $\Ind^G_B(\chi)$.} We define an locally $F$-analytic embedding $s:\bbP^1(F) \ra G$ by 
\[s(z) = \left(\begin{array}{cc} 1 & 0 \\
z & 1 \end{array}\right), \; \mbox{if } |z| \le 1 \,,\; 
s(z) = \left(\begin{array}{cc} 0 & -1 \\
1 & \frac{1}{z} \end{array}\right), \; 
\mbox{if } |z|>1 \;.\] 
Then, for every $z \in \bbP^1(F)$ and $g \in G$, one has $\xi(z,g) := s(z)gs(z.g)^{-1} \in B$ and
\begin{numequation}\label{cocycle} 
g = \xi(0,g)s(0.g) \,, \mbox{ and } \xi(z,gg') = \xi(z,g)\xi(z.g,g') \;.
\end{numequation}
\end{para}

\begin{lemma}\label{identification} The map

\[\zeta: \Ind^G_B(\chi) \lra C^\la(\bbP^1(F),E), \; \zeta(f)(z) = f(s(z)) \; \]

is an isomorphism of topological vector spaces. 
\end{lemma}
\begin{proof} Because $s$ and $f$ are locally $F$-analytic, so is $\zeta(f)$. We define another map $\tilde{\zeta}: C^\la(\bbP^1(F),E)) \ra \Ind^G_B(\chi)$ by $\tilde{\zeta}(f_1)(g) = \chi(\xi(0,g))f_1(s(0.g))$. We leave it to the reader to check that these maps are continuous when $\Ind^G_B(\chi)$ is equipped with the structure of a compact inductive limit, as explained in \ref{inducedrep}, and $C^\la(\bbP^1(F),E)$ is also considered as a vector space of compact type, cf.  \cite[Lemma 2.1]{S-T}. And it is easy to see that $\zeta$ and $\tilde{\zeta}$ are inverses of each other.
\end{proof}

Using $\zeta$ we equip $C^\la(\bbP^1(F),E)$ with a $G$-action. Explicitly, on $f \in C^\la(\bbP^1(F),E)$, the group action is given by 
\begin{numequation}\label{groupaction_1}(g._\chi f)(z) = \chi(\xi(z,g))f(z.g) \;.
\end{numequation}
If we write $\chi(\diag(a,d)) = \chi_1(ad)\chi_2(d)$, and if $g = \left(\begin{array}{cc} a & b \\
c & d \end{array}\right)$, then \ref{groupaction_1} becomes
\begin{numequation}\label{groupaction_2}(g._\chi f)(z) = \chi_1(ad-bc)\chi_2(bz+d)f\Big(\frac{az+c}{bz+d}\Big) \;.
\end{numequation}

\begin{lemma}\label{proper containment} Let $v$ be a vertex in $BT$, then for every orbit $\PD$ of $G_v(k)$ there is an
edge $e =\{v,v'\}$ and an orbit $\PD'$ of $G_{v'}(k)$ which is contained properly
in $\D$.
\end{lemma}

\begin{proof}
Without loss of generality we can assume that $v= v_0 := [\cO \oplus \cO]$. Then let $\PD = \bB_z(z_0, |\vpi|^k) \in \POmega_{v_0,k}$, if $z_0 \in \bB_z([\alpha], |\vpi|)$ for some $\alpha \in \bF_q$, then $G_{v_{\alpha, 1}}(k)$ has as an orbit $\bB_z(z_0, |\vpi|^{k+1})$ where $v_{\alpha, 1}$ is as in \ref{n=1}. If $\PD= \bB_w( w_0, |\vpi|^k) \subset \bB(\infty, |\vpi|)$ then by \ref{orbit} we have $\PD'= \bB_w(w_0, |\vpi|^{k+1})$ as an orbit of $G_{v_1}(k)$. 

\end{proof}

\begin{lemma} 
We have the following relations between $G_v(k)$-orbits for varying $v \in BT_0$. 

\vskip8pt

\label{containment}(i) Let $\{v,v'\} \in BT_1$ be an edge. Suppose there are $\PD \in \POmega_{v,k}$ and $\PD' \in \POmega_{v',k}$ such that $\oPD \subset \oPD'$. Then  for any edge $\{v'',v'\} \in BT_1$ with $d(v, v'')=2$ there is  $\PD'' \in \POmega_{v'',k}$ with $\oPD \subset \oPD' \subset \oPD''$.

\vskip8pt

\label{reversecontainment}(ii) Let $\{v',v''\} \in BT_{1}$ be an edge. Suppose there are $\PD' \in \POmega_{v',k}$ and $\PD'' \in \POmega_{v'',k}$ such that $\oPD' \subset \oPD''$. Then there exists some $\{v,v'\}\in BT_1$ with $d(v, v'') =2$ and $\PD \in \POmega_{v,k}$  such that $\oPD \subset \oPD' \subset \oPD''$.

\vskip8pt

\label{coveringminimal}(iii) Let $\{v',v''\} \in BT_1$ be an edge. Suppose that for some $\PD' \in \POmega_{v',k}$ and $\PD'' \in \POmega_{v'',k}$ we have $\oPD' \subset \oPD''$. Then 

$$\oPD' = \bigcup_{\substack{\{v',v\}\in BT_1, v\ne  v'' \\ \PD\in \POmega_{v,k} \\ \oPD\subset \oPD'}}\oPD$$

\end{lemma}

\begin{proof}
(i) By \ref{n-transitive} we can assume that $v''= v_0, v'= v_1$ and $v= v_2$ where $v_0= [\cO \oplus \cO]$ and $v_1=[(\vpi)\oplus \cO]$ as before and $v_2$ is defined as $v_2 := [ (\vpi^2) \oplus \cO]$.

\vskip8pt

We have already shown that the orbits of $G_{v_0}(k)$ on $\bbP^1(F)$ are discs of radius $|\vpi|^k$. On the other hand, the orbits of $G_{v_1}(k)$ are discs of radius $|\vpi|^{k-1}$ on  $\bB_z(0,1)$ and discs of radius $|\vpi|^{k+1}$ on $\bB_w(\infty, |\vpi|)$.  Similarly, orbits of $G_{v_2}(k)$ on $\bB_z(0, |\vpi|^{-1})$ are discs of radius $|\vpi|^{k-2}$ and on $\bB_w(\infty, |\vpi|^2)$  they are discs of radius $|\vpi|^{k+2}$. 

\vskip8pt

Therefore, if for $\PD \in \POmega_{v_2,k}, \oPD$ is contained in $\oPD'$ for  $\PD' \in \POmega_{v_1,k}$, then $\oPD$  is a disc of radius $|\vpi|^{k+2}$ and $\oPD'$ is a disc of  $|\vpi|^{k+1}$ containing it. But any such disc of radius $|\vpi|^{k+1}$ is contained in a disc of radius $|\vpi|^{k}$ which is an orbit of $G_{v_0}(k)$. This proves our claim.

\vskip8pt
(ii) For computational ease we pick $v'' = v_0 = [\cO \oplus \cO]$ and $v'= v_1 = [ (\vpi) \oplus \cO ]$. From \ref{particular orbits} we see that if $\oPD'  \subset \oPD''$ with $\PD' \in \POmega_{v',k}$ and $\PD'' \in \POmega_{v'',k}$ then $\oPD'$ is a disc of radius $|\vpi|^{k+1}$ inside $\bB_w(\infty, |\vpi|)$. Let $a \in F$ belong to $\oPD'$, then  choosing $v= [(\vpi^2,0), (a,1)]$ and noting $v= \begin{bmatrix} \vpi^2 & a \\ 0 & 1\end{bmatrix}. v_0 $, we can see that $\begin{bmatrix} \vpi^2 & a \\ 0 & 1\end{bmatrix}^{-1}$ transforms  $\bB_w(\infty, |\vpi|^k)$ to  $\PD\in \POmega_{v,k}$, a disc of radius $|\vpi|^{k+2}$ such that $\oPD \subset \oPD' \subset \oPD''$.
\vskip8pt
(iii) This is an extension of calculations of (ii), and it can be seen that the any disc of radius $|\vpi|^{k+1}$ that is an orbit of $G_{v_1}(k)$ is covered by a subset of discs of radius $|\vpi|^{k+2}$ which are orbits of various $G_v(K)$ with $\{v_1,v\}\in BT_{1,2}$.

\end{proof}

\begin{para}{\it Minimal Orbits.}
Recall that we partially order $\oOmega_{0,k,n}$ and $\oPOmega_{0,k,n}$ via inclusion. Then we define 
\[\oOmega_{0,k,n}^{\min}:=\big\{\oD \in \oOmega_{0,k,n}\big|\; \oD\; \mbox{ minimal w.r.t. the partial ordering of } \oOmega_{0,k,n}\big\}\]
and
\[\oPOmega_{0,k,n}^{\min}:=\big\{\oPD \in \oPOmega_{0,k,n}\big|\; \oPD\; \mbox{ minimal w.r.t. the partial ordering of } \oPOmega_{0,k,n}\big\} \;.\]
\end{para}

\begin{lemma}\label{minimalorbits} Fix $\PD\in \POmega_{v,k}$ with $v\in BT_{0,n}$, let $n\ge 1$.

\vskip8pt

(a) $\oPD \in \oPOmega_{0,k,n}$ belongs to $\oPOmega_{0,k,n}^{\min}$ iff both of the following conditions hold

\vskip8pt
(i) $d(v,v_0)= n$

\vskip8pt

(ii) Let $v_1$ be the unique vertex of $BT$ such that $\{v,v_1\} \in BT_{1,n}$. Then there exists $\PD_1\in \POmega_{v_1,k}$ such $\oPD \subset \oPD_1$.

\vskip8pt

(b) Condition (ii) above is equivalent to the condition

\vskip8pt

(ii)' For every vertex $v'\in BT_{0,n}$ if $v= v_1,\cdots, v_m= v'$ is the path connecting $v$ and $v'$ there exists $\PD_i \in \POmega_{v_i,k}$ such that $\oPD = \oPD_1 \subset \cdots \subset \oPD_m$.
\end{lemma}

\begin{proof}
(b) To see the equivalence of (ii)' and (ii) we first note that (ii)' definitely implies (ii).
\vskip8pt
Now assume (ii) holds, i.e., there are orbits $\PD\in \POmega_{v,k}$ and $\PD_1 \in \POmega_{v_1,k}$ such that $\oPD \subset \oPD_1$, where $v_1$ is the unique neighboring vertex of $v$ in $BT_{0,n}$. For any vertex $v'\in BT_{0,n}$ let the path connecting  $v, v'$ be $v, v_1, \cdots, v_m= v'$. Then by \ref{containment} we can find an orbit $\PD_2$ in $\POmega_{v_2,k}$ such that  $\oPD \subset \oPD_1\subset \oPD_2$. Continuing .this process we get a chain of nested orbits as in (ii)'

\vskip8pt

(a) We give a proof of (i) and (ii)' below.
\vskip8pt

First, suppose $\oPD$ is minimal in $\oPOmega_{0,k,n}$. Then, $d(v_0,v) = n$, because otherwise all adjacent
vertices $v'$ with $d(v,v')=1$ are also in $BT_{0,n}$, and hence, by \ref{proper containment}, any orbit of $G_v(k)$ properly contains an orbit of $G_{v'}(k)$ for some such $v'$. And thus is not minimal in $\Omega_{0,k,n}$. We thus have $d(v_0,v)=n$.

\vskip8pt

 Let $v_2$ be the unique vertex adjacent to $v=v_1$
and contained in $BT_{0,n}$. Then every orbit of $G_v(k)$ is contained in an
orbit of $G_{v_2}(k)$, or contains an orbit of $G_{v_2}(k)$. Since $\oPD=\oPD_1$ is
minimal, $\oPD$ is contained in $\oPD_2$ for an orbit $\oPD_2$ of $G_{v_2}(k)$.

Now let $v' \neq v$ be any vertex in $BT_{0,n}$ (the assertion is trivial
for $v'=v$). Let $v = v_1, v_2, ..., v_m = v'$ be the path from $v$ to $v'$. As we have seen, $\oPD_1$ is contained in  $\oPD_2$ for an orbit $\PD_2$ of $v_2$. By \ref{containment} we then find an orbit $\PD_3$ of $G_{v_3}(k)$ such that  $\oPD_2\subset \oPD_3$, and repeatedly applying \ref{containment} in the same way we conclude
that (ii)' holds true.

\vskip8pt

Conversely, assume  that $\PD$ satisfies assumption (ii)'. Suppose $\oPD$ properly contains
 $\oPD'$  for an orbit $\PD'$ of some $G_{v'}(k)$ with $v'$ in $BT_{0,n}$. Let $v = v_1,...,v_m = v'$ and $\oPD = \oPD_1 \subset \oPD_2 \subset ... \subset \oPD_m$ be as in (ii). Then $\oPD' \subsetneq \oPD \subset \oPD_m$ is a contradiction,
because $\PD'$ and $\PD_m$ are both orbits of $G_{v'}(k)$. Hence $\PD$ is minimal.

\end{proof}

\begin{prop} \label{minimalbijection}
If $v,v' \in BT_{0,n}$, and if $\PD \in \POmega_{v,k}$ and $\PD' \in {\rm P}\Omega_{v',k}$ are such that their images in $\overline{{\rm P}\Omega}_{0,k,n}$ are minimal and (whose underlying sets) are equal, then $v = v'$. 
\end{prop}

\begin{proof}

Let $\PD \in \POmega_{v,k}$ and $\PD' \in {\rm P}\Omega_{v',k}$ belong to  $\overline{{\rm P}\Omega}_{0,k,n}$  and assume $\oPD= \oPD'$. Let $v = v_1, v_2, ..., v_m = v'$ be the path from $v$ to $v'$. By \ref{minimalorbits} there exists $\PD_i \in \POmega_{v_i,k}$ such that $\oPD = \oPD_1 \subset \cdots \subset \oPD_m$. Note that $\oPD_1 \ne \oPD$ since by \ref{orbit} $\POmega_{v,k} \cap \POmega_{v_1,k}= \emptyset$. Hence, in particular $\oPD' \subsetneq \oPD_m \in \POmega_{v',k}$. Which is a contradiction.
\end{proof}

\begin{lemma}\label{minimalcovering}
 $\oPOmega_{0,k,n}^{\min}$ forms a disjoint covering of $\bbP^1(F)$ for all $n\geq 1$.
\end{lemma}

\begin{proof}
This is evidently true for $n=0$.
\vskip8pt
Assuming it is true for $n-1$, $n\ge 1,$ observe that it is enough to prove $\oPOmega_{0,k,n}^{\min}$ forms a covering of $\oPOmega_{0,k,n-1}^{\min}$. Let $\oPD'\in \oPOmega_{0,k,n-1}^{\min}$ then there exists a $\PD'' \in \POmega_{v'',k}$ with $\{v',v''\}\in BT_{1}$ such that $\oPD'\subset\oPD''$ (this is true for $n=1$ by \ref{orbit} and \ref{n=1}). Then by \ref{coveringminimal}
$\oPD'$ is covered by disjoint sets $\oPD\in \POmega_{v,k}$ with $\oPD\subset\oPD'\subset \oPD''$ where $v$ ranges over all vertices of $BT$ such that $\{v,v'\}\in BT_1, v\ne v''$. But such a $\oPD\in \oPOmega_{0,k,n}^{\min}$ by \ref{minimalorbits}. This proves our claim.

\end{proof}
\begin{rem}
The bijection between $\Omega_{v,k}$ and $\POmega_{v,k}$ naturally extends to $\Omega_{0,k,n}$ and $\POmega_{0,k,n}$, and the containment relations stated in this section about elements of $\POmega_{0,k,n}$ can be applied exactly in the same way to elements of $\Omega_{0,k,n}$. With this remark we will shift to working with $\Omega_{0,k,n}$ in the subsequent sections. Further, let $$\Omega_{0,k,n}^{\min}= \pr^{-1}(\oOmega_{0,k,n}^{\min})$$ 
Note that \ref{minimalbijection} applies to $\Omega_{0,k,n}$ and $\oOmega_{0,k,n}$ as well and $\pr: \Omega_{0,k,n}^{\min} \ra \oOmega_{0,k,n}^{\min}$ is a bijection.
\end{rem}

\begin{rem}\label{edgenotminimal}
We record the following containment relations

\vskip8pt

(i) From \ref{minimalorbits} we see that for any $v \in BT_{0,n-1}$ we have 

$$\Omega_{v,k} \cap \Omega_{0,k,n}^{\min}= \emptyset$$

\vskip8pt

(ii) From (iv) of \ref{orbit} and \ref{minimalorbits} it follows that for any $\D \in \Omega_{e,k}$ with $e= \{v, v'\}$ and $d(e, v_0) = n$,  
$$\Omega_{e,k} \cap  \Omega_{0,k,n}^{\min} = \emptyset$$

\end{rem}

\subsection{Detailed description of rigid analytic vectors}

We will need the following lemma to prove an important detail in the subsequent lemma.

\begin{lemma}\label{rigidanalyticfunctions} Let $\bbH$ be an affinoid rigid analytic group over $F$, and assume that $\bbH$ is isomorphic as a rigid analytic space to $\bbB^d$ via a chart $x: \bbH \ra \bbB^d$. Let $\bbU$ be an affinoid rigid analytic space which is equipped with a rigid analytic action of $\bbH$ from the right. Furthermore, we assume that there is $z_0 \in U = \bbU(F)$ and a closed rigid analytic subgroup $\bbH' \sub \bbH$ such that the map $\bbH' \ra \bbU$, $h \mapsto z_0.h$, is an isomorphism of rigid analytic spaces. \\

We let $H$ act on $V = C^\la(U,E)$ by $(g.f)(z) = f(zg)$. Then $f \in V$ is $\bbH$-analytic if and only if $f \in \cO(\bbU)$.  
\end{lemma}

\begin{proof} (i) If $f \in \cO(\bbU)$, then we consider $f$ as a rigid analytic morphism $f: \bbU \ra \bbA^{1,\rig}$. Because the group action $\mu: \bbU \times \bbH \ra \bbU$ is a rigid analytic morphism, so is $f \circ \mu$, and hence $f \circ \mu \in \cO(\bbU \times \bbH) = \cO(\bbU) \widehat{\otimes}_F \cO(\bbH) = \cO(\bbU)\langle x_1, \ldots, x_d\rangle$, where the latter is the ring of strictly convergent power series over the Banach algebra $\cO(\bbU)$. Then, for $h \in \bbH$ and $z \in \bbU$ we have \[(h.f)(z) = f(z.h) = (f \circ \mu)(z,h) = \sum_{\nu \in \bbN^d} f_\nu(z) x(h)^\nu \;.\] 
And hence $f$ is $\bbH$-analytic.

\vskip8pt

(ii) Conversely, if $f$ is $\bbH$-analytic, then we can write 

\[h.f = \sum_{\nu \in \bbN^d} f_\nu x(h)^\nu\]

where all functions $f_\nu$ are in a single BH-subspace of $C^\la(U,E)$. This means that there is a finite covering $(U_i)_{i = 1}^n$ of $U$ consisting of disjoint compact open subsets. After refining this covering, we may assume that each $U_i$ is the set of $F$-valued points of an afinoid subspace $\bbU_i \sub \bbU$. Put $\bbU' = \coprod_{i=1}^n \bbU_i$, which is again an affinoid subdomain of $\bbU$, and which has the property that $\bbU'(F) = \bbU(F) \, (= U)$. The functions $f_\nu$ are then all in $\cO(\bbU')$, and we have 

\[\|f_\nu\|_{\bbU'} \, \cdot \, \sup\{|x(h)^\nu| \midc h \in \bbH\} \; \lra \; 0 \;\; \mbox{as} \;\; |\nu| \ra \infty \;.\]
If $z_0 \in U =  \bbU'(F)$ is as in the statement of the lemma, we find for all $h \in H$:

\begin{equation}\label{expansion}
    f(z_0 h) = \sum_\nu f_\nu(z_0) x(h)^\nu
\end{equation}

The right hand side of \ref{expansion} makes sense and converges for all $h \in \bbH$, and is thus a rigid analytic function on $\bbH$. When we restrict the right hand side of \ref{expansion} to $\bbH'$ it is hence rigid analytic on $\bbH'$. Let $\alpha: \bbH' \ra \bbU$, be defined by $h \mapsto z_0.h$. Then we find that $f \circ \alpha$ is a rigid analytic map on $\bbH'$, and $f = (f \circ \alpha) \circ \alpha^{-1}$ is rigid analytic on $\bbU$. 
\end{proof}

\begin{lemma}\label{samerigid}
There exists $k_0  = k_0(\chi) \in \bZ_{\ge 1}$ such that for all $k \ge k_0$ the following statements holds. If $e=\{v,v'\}$ is a vertex and $\D \in \Omega_{e,k}\cap \Omega_{v,k}$, then  
\[I(\D,\chi)_{\bbG_e(k)-\an} = I(\D, \chi)_{\bbG_v(k)-\an}\]
\end{lemma}
\begin{proof} Let $\D \leftrightarrow \PD$ be the bijection $\Omega_{v,k} \leftrightarrow \POmega_{v,k}$, as in \ref{Omega-ks}. We set $I(\PD) = C^\la(\PD,E)$. Under the bijection $\zeta$ of \ref{identification}, the space $I(\D,\chi)$ gets mapped isomorphically to $I(\PD)$. The latter carries the action of $G_e(k)$ (and hence of $G_v(k) \sub G_e(k))$ defined by \ref{groupaction_2}. With respect to this group action, it suffices to show that 
\begin{numequation}\label{equality}
    I(\PD)_{\bbG_e(k)-\an} = I(\PD)_{\bbG_v(k)-\an}
\end{numequation}
From the transitivity of the $G$-action on $BT_0$ and $BT_1$ it is easy to see given any two pairs $(v_1, e_1), (v_2, e_2)$ with $v_i \in BT_0$ and $e_i \in BT_1$ there is a $g\in G$ such that $(g.v_2, g.e_2)= (v_1, e_1)$. It follows that $G_{e_1}(k) = gG_{e_2}(k)g^{-1}$ and $G_{v_1}(k) = gG_{v_2}(k)g^{-1}$. From \cite[3.5.1]{Em17} it follows that for $\PD \in \Omega_{e_1,k} \cap \Omega_{v_1,k}$ one has
\begin{numequation}\label{equality2}
g.\Big(I(\PD)_{\bbG_{e_1}(k)-\an}\Big) =  I(\PD.g)_{\bbG_{e_2}(k)-\an}
\end{numequation}
and
\begin{numequation}\label{equality3}
g.\Big(I(\PD)_{\bbG_{v_1}(k)-\an}\Big) =  I(\PD.g)_{\bbG_{v_2}(k)-\an} \;.
\end{numequation}
In particular, it suffices to show \ref{equality} for a single pair $(v,e)$ which we are free to choose. The following choice is particularly convenient: $e = \{v,v'\}$ with $v = [\cO \oplus \cO]$ and $v' = [(\vpi)\oplus \cO]$.  By computations in \ref{orbit} any $\PD \in \POmega_{v,k}\cap \POmega_{e,k}$ is of the form $\bB_z(a, |\vpi|^k)$ with $|a| \leq 1$.  Now action of $G_e(k)$ and $G_v(k)$ on $I(\PD)$ is given by

\[(g._\chi f)(x) = \chi_1(\det g)\chi_2(bx+d) f(z.g)\]

Note that for all $g \in G_e(k)$ one has $|\det g - 1 | \leq |\vpi|^k$, (and automatically for $g \in G_v(k)$, since $G_v(k) \subset G_e(k)$) and for all $|x| \leq 1$ and for all $g \in G_e(k)$ one has $|bx+d -1| \leq |\vpi|^k$. Thus for large enough $k$ by local analyticity of $\chi_1$ and $\chi_2$ we can make sure that both $\chi_1(\det g)$ and $\chi_2(bx+d)$ are expressible as a power series for $x \in \bB_z(a, |\vpi|^k)$ and $g\in G_e(k)$. Now we see from \ref{rigidanalyticfunctions} that  both $I(\PD,\chi)_{\bbG_e(k)-\an}$ and $I(\PD, \chi)_{\bbG_v(k)-\an}$ are same as $\cO(\PD)$. Here we have used \ref{rigidanalyticfunctions} twice with $\bbH = \bbG_e(k)$ and $\bbH=\bbG_v(k)$, respectively. This finishes the proof.
\end{proof}

\subsection{Surjectivity of \texorpdfstring{$\partial_0$}{}}

\begin{prop}\label{surjective}
There exists $k_0 = k_0(\chi) \in \bZ_{\ge 1}  $ such that for all $k \ge k_0$, the map $\partial_0: \bigoplus_{v\in BT_0} V_v \longrightarrow V$ is surjective.
\end{prop}

\begin{proof}
Let $f \in C^\la(\bbP^1(F), E)$. Let us partition $\bbP^1(F)$ into finitely many discs such that $f$ restricted to each of these discs is a rigid analytic function on that disc. Let $\alpha \in \bbP^1(F)$ and $\bB_{\alpha}$ be the disc containing $\alpha$ such that $f|_{\bB_\alpha} \in \cO(\bB_\alpha)$. Let $g \in G$ be such that $\alpha = 0. g$ and let $0 \in U$ be a neighborhood around $0$ such that $U.g = \bB_\alpha$. For $v_n:= [\cO \oplus (\vpi^n)]$ we can show that $\bB_z( 0, |\vpi|^{k+n}) \in \POmega_{v_n, k}$. Thus for any $k \ge 0$ there exists $n \ge 0$ such that $\bB_z(0, |\vpi|^{k+n}) \subset U$. By \ref{groupaction_2} the group action is given by the formula

\[(g._\chi f)(z) = \chi_1(\det g)\chi_2(bz+d) f(z.g)\]

Note that for all $g \in G_{v_n}(k)$ and for all $z \in \bB(0,|\vpi|^{k+n})$ one has $|\det g - 1 | \leq |\vpi|^k$,  and  $|bz+d -1| \leq |\vpi|^k$.

\vskip8pt

Since $f(z.g)$ is rigid analytic in $\bB( 0, |\vpi|^{k+n})$ , by \ref{rigidanalyticfunctions} we see that for large enough $k= k_0(\chi)$ we have  $(g._\chi f)|_{\bB( 0, |\vpi|^{k+n})}\in V_{v_n}$, note that $k_0(\chi)$ depends only on $\chi$ and not on the choice of $f$ . Let $U_\alpha = \bB_z( 0, |\vpi|^{k+n}). g $, then it is clear that $U_\alpha \in \POmega_{g. v_n, k}$ and by \ref{equality3} it follows that  $f|_{U_\alpha} \in V_{g. v_n}$. For each $\alpha \in \bbP^1(F)$ we find such a neighborhood $U_\alpha$ and assign a unique vertex $v_\alpha$ using the process above with $f|_{U_\alpha} \in V_{v_\alpha}$. With this setup, for the covering $\bbP^1(F)= \cup_\alpha U_\alpha$, there is a disjoint finite sub-covering $\cD$ such that if we define

\begin{numequation}\label{surjf} f_v = \left\{\begin{array}{ccl} f|_{U_\alpha} &, & v = v_\alpha, U_\alpha \in \cD \\
0 & ,&  \mbox{otherwise.} \end{array}\right.
\end{numequation}

Then we have
$$\partial_0( (f_v)_v) = f $$
\end{proof}

\subsection{Counting Arguments}

\begin{lemma}\label{cardorbits}
For any  $v \in BT_0$ and $e\in BT_1$ (or $BT_1^{\ori}$)  we have  $$|\Omega_{v,k}| = q^{k-1}+ q^k= (q+1)q^{k-1}, \;\;|\Omega_{e,k}|= 2q^{k-1}.$$
\end{lemma}
\begin{proof}
By \ref{cosetbijection}, the sets $\Omega_{v,k}$ and $\Omega_{v',k}$, for any two vertices $v, v' \in BT_0$, have the same cardinality. By \ref{Omega-ks}, $\Omega_{v,k}$ and $\POmega_{v,k}$ have the same cardinality, hence it is enough to find $|\POmega_{v_0,k}|$ where $v_0= [\cO \oplus \cO]$ as above. From the  descriptions as given in \ref{orbit} we see that the orbits of $G_{v_0}(k)$ in both $\bB_z(0,1)$ and $\bB_w(\infty, |\vpi|)$ are balls of radius $|\vpi|^k$. Thus $|\POmega_{v_0,k}|= \frac{1}{|\vpi|^k}+ \frac{|\vpi|}{|\vpi|^k} = q^k+ q^{k-1}$. Similarly, we compute $|\POmega_{e,k}|$, where $e=\{v_0, v_1\}$ with $v_1=[(\vpi)\oplus\cO]$ by the explicit description in \ref{orbit}.
\end{proof}
\begin{lemma}\label{cardmin}
Given a vertex $v \in BT_0$ with $d(v,v_0)= n$, 

$$|\Omega_{v,k}\cap \Omega_{0,k,n}^{\min}| = q^k$$
\end{lemma}
\begin{proof}
Given such a vertex $v$, let $v'$ be the unique vertex in $BT_{0, n-1}$ such that $\{v,v'\}\in BT_{1,n}$. Then by \ref{minimalorbits} $\D \in \Omega_{v,k}$ belongs to $\Omega_{0,k,n}^{\min}$ iff $\oD \subset \oD'$ for some $\D' \in \Omega_{v',k}$. By transitivity of $G$-action on $BT_1$, we can look at the vertices $v_0,v_1$ constituting the edge $e=\{v_0,v_1\}$ where $v_0, v_1$ are as in \ref{simplexgroups}. By \ref{orbit} (ii), $\Omega_{v',k}^{v} := \{\D' \in \Omega_{v',k} \midc \exists \D \in \Omega_{v,k}: \D \sub \D'\}$ has cardinality $q^{k-1}$ and each $\D' \in \Omega_{v',k}^{v}$ contains $q$ orbits of $\Omega_{v,k}$.
\end{proof}

Recall that we have remarked in \ref{subsetrelations} that $\Omega_{1,k} \subset \Omega_{0,k}$ and in fact it is easy to see that this induces an inclusion $\Omega_{1,k,n} \subset \Omega_{0,k,n}$ on these subsets. With this identification in place we claim,
\begin{prop}\label{cardsame}
\begin{align*}
 \Omega_{0,k,n}\setminus \Omega_{0,k,n}^{\min} =\Omega_{1,k,n}.\end{align*}
\end{prop}

\begin{proof} Recall that from \ref{edgenotminimal} it follows that $\Omega_{1,k,n} \cap \Omega_{0,k,n}^{\min}= \phi$. Thus, 
from the remark just before the proposition, it is enough to prove  $\big|\Omega_{0,k,n}\setminus \Omega_{0,k,n}^{\min}\big| =\big|\Omega_{1,k,n}\big|$. First we compute $|\Omega_{1,k,n}|$.
\[\begin{array}{lllll}
|\{e \in BT_1^{\ori}| d(v_0,e)= i\}| & = &(q+1). q^{i-1}\\
|BT_{1,n}^{\ori}|  & = & \sum_{i=1}^{n} |\{e \in BT_1^{\ori}| d(v_0,e)= i\}|\\ & = &(q+1). q^0 + \cdots + (q+1)q^{n-1}\\ &= & (q+1) \frac{q^n-1}{q-1} \\
|\Omega_{e,k}| &= &2 q^{k-1}\\
\big|\Omega_{1,k,n} \big| &=&  |\Omega_{e,k}|\times |BT_{1,n}^{\ori}| = 2 q^{k-1} (q+1) \frac{q^n-1}{q-1}
\end{array}\]
where the last but one formula is \ref{cardorbits}. To find $|\Omega_{0,k,n}\setminus \Omega_{0,k,n}^{\min}|$ first we collect the following (where here $i \ge 1$ in the first formula)
\[\begin{array}{lclc}
|\{v \in BT_0| d(v_0,v)= i\}| & = &(q+1). q^{i-1}\\
|BT_{0, n-1}| &=& 1 + (q+1)q^0+ \cdots + (q+1)q^{n-2}\\
|\Omega_{0,k,n}^{\min}\cap \Omega_{v,k}| & = & 0  \;\;\;\;\;\;\text{if}\;\;\;d(v_0,v) \leq n-1  \;\; (\ref{edgenotminimal}\mbox{ (i)})\\
|\Omega_{0,k,n}^{\min}\cap \Omega_{v,k}| &=& q^k \;\;\;\;\text{if} \;\; \;d(v_0,v)= n \;\; (\ref{cardmin})
\end{array}\]

Putting all these together we get 
\[\begin{array}{lll}
&&\big|\Omega_{0,k,n}\setm \Omega_{0,k,n}^{\min} \big| = (q+1)q^{k-1}|BT_{0,n-1}| + q^{k-1}|\{v \in BT_0| d(v_0,v)= n\}|\\ &=& (q+1)q^{k-1}(1 + (q+1)q^0+ \cdots + (q+1)q^{n-2}) + q^{k-1}. (q+1). q^{n-1} \\ &=& (q+1)q^{k-1}(1 + q + 1 + q^2+q+\cdots + q^{n-1}+ q^{n-2}+ q^{n-1}) \\ &=& 2q^{k-1}(q+1)\frac{q^n-1}{q-1}
\end{array}\]

\end{proof}

\subsection{Exactness in the middle}

Let,

\begin{align*}
C_{0,n,k} &= \bigoplus_{\substack{\D \in \Omega_{v,k}\\ v\in BT_{0,n}}} I(\D,\chi)_{\bbG_{v}(k)-\an},\\ C_{1,n,k}^\ori &= \bigoplus_{\substack{\D \in \Omega_{e,k}\\ e\in BT_{1,n}^\ori}} I(\D,\chi)_{\bbG_{e}(k)-\an}
\end{align*}

In order to show that the sequence \ref{chaincomp} is exact in the middle, it is enough to show that \[0 \ra C_{1,n,k}^\ori \xrightarrow{\partial_1} C_{0,n,k} \xrightarrow{\partial_0} V \ra 0\] is exact in the middle for every $n\geq 1$, where $\partial_1$ and $\partial_0$  denote what technically are restrictions of those maps to $C_{1,n,k}^\ori$ and $C_{0,n,k}$ respectively.

\vskip8pt

\begin{para}\label{subspace} {\it Some sub-spaces of $C_{0,n,k}$.} We put
\[C_{0,n,k}^{\min} = \bigoplus_{\substack{\D \in \Omega_{0,k,n}^{\min}}} I(\D,\chi)_{\bbG_{v}(k)-\an}\] 
and
\[C_{0,n,k}^{\nonmin} = \bigoplus_{\substack{\D \in \Omega_{0,k,n}\setminus \Omega_{0,k,n}^{\min}}} I(\D,\chi)_{\bbG_{v}(k)-\an} \;.\]
In the following, we write an element 
\[f_v \in V_v = \bigoplus_{\D' \in \Omega_{v,k}} I(\D',\chi)_{\bbG_v(k)-\an}\] 
as $f_v = (f_{\D', v})_{\D'\in \Omega_{v,k}}$. 

\end{para} 

\begin{lemma} \label{resrigid}

Given $k>0$  let $v, v'$ be vertices in $BT_{0,n}$ and let $\D \in \Omega_{v,k}$ and $\D' \in \Omega_{v',k}$ be such that $\oD \subset \oD'$ and $\D \in \Omega_{0,k,n}^{\min}$. Then for $f \in I(\D',\chi)_{\bbG_v'(k)-\an}$ we have $f|_{\D} \in I(\D,\chi)_{\bbG_{v}(k)-\an}$

\end{lemma}

\begin{proof}
Let us first note that for an edge $e= \{v, v'\} \in BT_1$ such that $\oD \subset \oD'$ with $\D' \in \Omega_{v',k}$ and $\D \in \Omega_{v,k}$ we have $\D' \in \Omega_{e,k}$ (by \ref{orbit} (v)). Now by \ref{samerigid} we have $I(\D', \chi)_{\bbG_v'(k)-\an}= I(\D', \chi)_{\bbG_e(k)-\an}$. For $f\in I(\D', \chi)_{\bbG_{v'}(k)-\an} = I(\D, \chi)_{\bbG_e(k)-\an} \subset V_e $ we have $f|_{\D} \in I(\D, \chi)_{\bbG_v(k)-\an}$, since $V_{e} \hookrightarrow V_{v}$.

\vskip8pt

Now let $d(v,v')>1$ and let $v= v_0, \cdots, v_m= v'$ be the path from $v$ to $v'$. By \ref{minimalorbits} we have, for $\D \in \Omega_{0,k,n}^{\min}$ a nested sequence of double cosets (seen as a subset of $G$) $\oD =\oD_0\subset \oD_1 \subset \oD_{n-1}\subset \cdots \subset \oD_m =\oD'$ with $\D_{i}\in \Omega_{v_i,k}$. Let $e_i:=\{v_i, v_{i+1}\} \in BT_1$. Since  $\oD_i\subset \oD_{i+1}$ by \ref{orbit} (v)  we have that $\D_{i+1}\in \Omega_{e_i,k}$. Thus $I(\D_{i+1}, \chi)_{\bbG_{e_i(k)}-\an}= I(\D_{i+1}, \chi)_{\bbG_{v_{i+1}(k)}-\an}$ and we successively get that $f|_{\D_i} \in I(\D_i,\chi)_{\bbG_{v_i}(k)-\an}$ for each $i$. Proving our claim.
\end{proof}

\begin{lemma}\label{kernel} The projection of 
\begin{numequation}\label{kerpartial}
C_{0,n,k} = C_{0,n,k}^{\min} \oplus C_{0,n,k}^{\nonmin}\end{numequation}
onto $C_{0,n,k}^{\nonmin}$ maps ${\rm Ker}(C_{0,n,k}\xrightarrow{\partial_0} V)$ isomorphically onto $C_{0,n,k}^{\nonmin}$.
\end{lemma}

\begin{proof}
By \ref{minimalcovering} the union of the minimal orbit in $\oPOmega_{0,k,n}$ is equal to $\bbP^1(F)$. Hence, the union of the minimal double cosets in $\Omega_{0,k,n}$ is equal to $G$. Now, if $(f_v)_{v \in BT_{0,n}}$ is in $\ker(\partial_0)$, then $\sum_{v\in BT_{0,n}} \sum_{\D' \in \Omega_{v,k}} f_{\D',v} = 0$. Now we restrict both sides of this equation to a minimal double coset $\D \in \Omega_{0,k,n}^{\min}$ and obtain
\[\sum_{\D' \in \Omega_{0,k,n}} \sum_{\D' \supset \D} f_{\D',v}|_{\D} = 0\;,\] equivalently 
\[f_{\D,v} = -\sum_{\D' \in \Omega_{0,k,n}} \sum_{\D' \supsetneq \D} f_{\D',v'}|_{\D} \;.\]
This shows that the components $f_{\D,v}$ with $\D \in \Omega_{v,k}$ and in  $\Omega_{0,k,n}^{\min}$ are uniquely determined by the other components. Note that by \ref{resrigid} we see that $f_{\D',v'}|_{\D}\in I(\D, \chi)_{\bG_v(k)-\an}$ for any $\D \in \Omega_{v,k}$ such that $\D \in \Omega_{0,k,n}^{\min}$ and $\oD\subset \oD'$. So the equation is well-defined in particular. Conversely, for any element \[(f_{\D',v})_{\D' \in \Omega_{0,k,n}\setminus  \Omega_{0,k,n}^{\min}} \in C_{0,n,k}^{\nonmin}\] 
we obtain an element $(f_v)_v$ of $\ker(\partial_0)$ by defining the components $f_{\D,v}$ for $\D \in \Omega_{v,k}$ and in $\Omega_{0,k,n}^{\min}$ by the previous equation, the equation being well defined by \ref{resrigid}. This proves our claim.
\end{proof}

Let $p: C_{0,n,k}\ra \ker(\partial_0)$ be the projection on to the second summand as in \ref{kerpartial}. With these identifications we write the composition map as 

$$C_{1,n,k}^\ori \xrightarrow{\partial_1} C_{0,n,k} \xrightarrow{p} \ker(\partial_0)$$ 

and call it, $\opartial_1 := p\circ \partial_1$.

\begin{prop}\label{middleexact} There exists some $k_0(\chi) \in \Z_{>0}$ such that for all $k\ge k_0(\chi)$ the composite map $\opartial_1= p\circ \partial_1$ is an isomorphism from $C_{1,k,n}^{\ori}$ onto $\ker(\partial_0)$.
\end{prop}

\begin{proof}

We refine the partial ordering on $\oOmega_{0,k,n}$ to a total ordering on  $\Omega_{0,k,n}\setminus\Omega_{0,k,n}^{\min}$.
By \ref{samerigid} and \ref{cardsame} $C_{1,k,n}$  maps isomorphically onto $C_{0,k,n}^{\nonmin}$ for all $k\ge k_0(\chi)$ for some $k_0(\chi)\in \Z_{>0}$. Thus we can view (for $k\ge k_0(\chi)$) $\opartial_1$ explicitly as a map between

$$\left[\bigoplus_{\substack{\D \in \Omega_{0,k,n}\setminus \Omega_{0,k,n}^{\min}}} I(\D,\chi)_{\bbG_{v}(k)-\an}\right] \stackrel{\opartial_1}{\longrightarrow}\left[\bigoplus_{\substack{\D \in \Omega_{0,k,n}\setminus \Omega_{0,k,n}^{\min}}} I(\D,\chi)_{\bbG_{v}(k)-\an}\right]$$

\vskip8pt

Let us define  for $e =\{v_1, v_2\}\in BT_1$, $\sgn^e_{v_i}= \begin{cases} 1 \;\;\; \;\;\mbox{if}\; (v_i, v_j) \in BT_1^{\ori}\\ -1 \;\;\; \mbox{if}\; (v_j, v_i) \in BT_1^{\ori}\end{cases}$.

\vskip8pt 

Given $f_{\D,v} \in I(\D, \chi)_{\bbG_e-\an}= I(\D, \chi)_{\bbG_v-\an}$ for some $e= \{v,v'\}\in BT_{1,n}$ and $\D \in \Omega_{e,k}\cap \Omega_{v,k}$ indexing elements of  $ \bigoplus_{\substack{\D \in \Omega_{0,k,n}\setminus \Omega_{0,k,n}^{\min}}} I(\D,\chi)_{\bbG_{v}(k)-\an} $ as $(f_{\D, v})_{\D, v}$ we see that,

\begin{numequation}\label{opartial} \begin{array}{cl} & \opartial_1(0, \cdots, f_{\D,v}, \cdots, 0) \\
&\\
= & (g_{\D_\alpha, v_\alpha})_{\D_\alpha, v_\alpha}\\
&\\
= & \left\{\begin{array}{ccl} \sgn^e_v. f_{\D, v} & , & \mbox{if}\; \; \D_\alpha= \D, v_\alpha= v, \\
\sgn^e_v. f_{\D, v}|_{\D'} & , & \mbox{if}\; \; \D_\alpha= \D', v_\alpha= v', \oD'\subset\oD\\
0 & , & \mbox{otherwise} \end{array}\right.
\end{array}
\end{numequation}
 
This shows that $\opartial_1$ can be expressed as a lower triangular $r \times r$ matrix with $\pm Id$ on the diagonal and $\pm\res^{\D}_{\D'}$ maps on the off-diagonal elements which corresponds to $((\D, v), (\D', v'))$ such that $\D'\subset \D$ and $\{v, v'\} \in BT_1$ and $0$ elsewhere, where $r= 2q^{k-1}(q+1) \frac{q^n-1}{q-1}$ and $\res^{\D}_{\D'}(f_{\D,v}) := g_{\D', v'}= f_\D|_{\D'}$. Thus $\opartial_1$ is an isomorphism as claimed.
\end{proof}

The exactness of the chain complex \ref{chaincomp} for $V:= \Ind^G_B(\chi)$ now follows from  \ref{injective}, \ref{surjective} and \ref{middleexact}.

\subsection{An example of the matrix of \texorpdfstring{$\opartial_1$}{}}
 We give an example of the matrix of $\opartial_1$ in the case of $n=1, k=1$. Let $v^0, v^1,v^\infty$ be the vertices adjacent to $v_0$. Let $\D_0$, $\D_1$, $\D_\infty$ be the $B$-$G_{v_0}(1)$ double cosets. Let $e_0, e_1, e_\infty$ be the oriented edges connecting $v_0$ (origin) with $v^0, v^1,v^\infty$, respectively. We may assume that $\D_i$ and $\D_{j \cup k} := \D_j \cup \D_k$ are the two $B$-$G_{e_i}(1)$ double cosets of $G_{e_i}(1)$, where $\{i,j,k\} = \{0,1,\infty\}$. The group $G_{v^i}(1)$ has then the $B$-$G_{v^i}(1)$ double cosets $\D^i_0, \D^i_1$, whose union is $\D_i$, and the double coset $\D_{j \cup k}$ for $j \neq i$, $k \neq i$ (and $j \neq k$). We then have 
\[\Omega_{0,1,1}^{\min} = \{\D^0_0, \D^0_1, \D^1_0, \D^1_1,\D^\infty_0, \D^\infty_1\}\]
and 
\[\Omega_{0,1,1}^{{\nonmin}} = \{\D_0, \D_1, \D_\infty, \D_{0 \cup 1},\D_{0 \cup \infty}, \D_{1 \cup \infty}\}\]
Let $p$ be the projection from $C_{0,n,k}$ to $C_{0,n,k}^{\nonmin}$, and set $\opartial_1 = p \circ \partial_1$. We write

\[\begin{array}{rcl}C_{1,1,1}^\ori & = &  I(\D_{1 \cup \infty},\chi)_{G_{e_0}-\an}  \oplus I(\D_{0 \cup \infty},\chi)_{G_{e_1}-\an} \oplus  I(\D_{0\cup 1},\chi)_{G_{e_\infty}-\an}\\
&& \oplus I(\D_0,\chi)_{G_{e_0}-\an}   \oplus I(\D_1,\chi)_{G_{e_1}-\an}
 \oplus I(\D_\infty,\chi)_{G_{e_\infty}-\an} 
 \end{array}\]
 
 \[\begin{array}{rcl}C_{0,1,1}^{\nonmin} & = &  I(\D_{1 \cup \infty},\chi)_{G_{v_{0,0}}-\an}  \oplus I(\D_{0 \cup \infty},\chi)_{G_{v_{0,1}}-\an} \oplus  I(\D_{0\cup 1},\chi)_{G_{v_{0,\infty}}-\an}\\
&& \oplus I(\D_0,\chi)_{G_{v_0}-\an}   \oplus I(\D_1,\chi)_{G_{v_0}-\an}
 \oplus I(\D_\infty,\chi)_{G_{v_0}-\an} 
 \end{array}\]
 The the matrix of $\opartial_1$ is given by
 \[\left(\begin{array}{cccccc} 
-1 & 0 & 0& 0 & 0 & 0 \\
0  & -1 & 0& 0 & 0 & 0 \\
0  & 0 &  -1& 0 & 0 & 0 \\
0  & \res^{0 \cup \infty}_0 & \res^{0 \cup 1}_0& 1 & 0 & 0 \\
\res^{1 \cup \infty}_1& 0 &  \res^{0 \cup 1}_1& 0 & 1 & 0 \\
\res^{1 \cup \infty}_\infty& \res^{0 \cup \infty}_\infty & 0& 0 & 0 & 1 \\
\end{array}\right)\]
where $\res^{i \cup j}_i(f_{\D_{i\cup j, v^k}}):= g_{\D_i, v_0}= f_{\D_{i\cup j}}|_{\D_i}$.

\bibliographystyle{abbrv}
\bibliography{mybib}

\begin{thebibliography}{1}

\bibitem{DDMS}
J.~D. Dixon, M.~P.~F. du~Sautoy, A.~Mann, and D.~Segal.
\newblock {\em Analytic pro-{$p$}-groups}, volume 157 of {\em London
  Mathematical Society Lecture Note Series}.
\newblock Cambridge University Press, Cambridge, 1991.

\bibitem{Em17}
M.~Emerton.
\newblock Locally analytic vectors in representations of locally {$p$}-adic
  analytic groups.
\newblock {\em Mem. Amer. Math. Soc.}, 248(1175):iv+158, 2017.

\bibitem{La}
A.~Lahiri.
\newblock {R}igid analytic vectors in locally analytic representations.
\newblock {\em { {A}nnales math\'ematiques du {Q}u\'ebec. {A}vailable at
  \url{https://arxiv.org/pdf/1907.12220.pdf}}}, 2020.

\bibitem{SS1}
P.~Schneider and U.~Stuhler.
\newblock Resolutions for smooth representations of the general linear group
  over a local field.
\newblock {\em J. Reine Angew. Math.}, 436:19--32, 1993.

\bibitem{SS2}
P.~Schneider and U.~Stuhler.
\newblock Representation theory and sheaves on the {B}ruhat-{T}its building.
\newblock {\em Inst. Hautes \'{E}tudes Sci. Publ. Math.}, (85):97--191, 1997.

\bibitem{S-T}
P.~Schneider and J.~Teitelbaum.
\newblock Locally analytic distributions and {$p$}-adic representation theory,
  with applications to {${\rm GL}_2$}.
\newblock {\em J. Amer. Math. Soc.}, 15(2):443--468, 2002.

\end{thebibliography}

\end{document}